\documentclass[reqno]{amsart}

\usepackage[utf8]{inputenc} 
\usepackage{amssymb}
\usepackage{mathrsfs}
\usepackage{graphicx}
\usepackage{latexsym}
\usepackage{stmaryrd}
\usepackage{enumitem}
\usepackage[bookmarks,hyperfootnotes=false,backref]{hyperref}
\usepackage{multirow}
\usepackage{xcolor}
\usepackage{caption}
\colorlet{darkishRed}{red!80!black}
\colorlet{darkishBlue}{blue!60!black}
\colorlet{darkishGreen}{green!60!black}
\hypersetup{
    draft = false,
    bookmarksopen=true,
    colorlinks,
    linkcolor={red!60!black},
    citecolor={green!60!black},
    urlcolor={blue!60!black}
}
\usepackage[msc-links,backrefs]{amsrefs} 
\usepackage{doi}

\renewcommand{\PrintDOI}[1]{\doi{#1}}

\let\setminus=\smallsetminus
\usepackage{tikz}
\usepackage{tikz-cd}
\usetikzlibrary{calc,through,intersections,arrows, trees, positioning, decorations.pathmorphing, cd}
\usepackage{relsize}
\usepackage{comment}
\usepackage{svg}
\usepackage{mathtools}
\usepackage{nccmath}
\usepackage{pifont}

\let\setminus=\smallsetminus

\newcommand{\invLim}{\varprojlim}

\newcommand{\rest}{\upharpoonright}

\renewcommand{\subset}{\subseteq}
\renewcommand{\supset}{\supseteq}

\newcommand{ \N } { \mathbb{N} }

\newcommand{ \Q } { \mathbb{Q} }

\makeatletter

\def\calCommandfactory#1{%
   \expandafter\def\csname c#1\endcsname{\mathcal{#1}}}
\def\frakCommandfactory#1{%
   \expandafter\def\csname frak#1\endcsname{\mathfrak{#1}}}

\newcounter{ctr}
\loop
  \stepcounter{ctr}
  \edef\X{\@Alph\c@ctr}
  \expandafter\calCommandfactory\X
  \expandafter\frakCommandfactory\X
  \edef\Y{\@alph\c@ctr}
  \expandafter\frakCommandfactory\Y
\ifnum\thectr<26
\repeat

\renewcommand{\cP}{\mathscr{P}}

\setenumerate{label={\normalfont (\roman*)},itemsep=0pt}

\def\lowfwd #1#2#3{{\mathop{\kern0pt #1}\limits^{\kern#2pt\raise.#3ex
\vbox to 0pt{\hbox{$\scriptscriptstyle\rightarrow$}\vss}}}}
\def\lowbkwd #1#2#3{{\mathop{\kern0pt #1}\limits^{\kern#2pt\raise.#3ex
\vbox to 0pt{\hbox{$\scriptscriptstyle\leftarrow$}\vss}}}}

\def\ve{\kern-1.5pt\lowfwd e{1.5}2\kern-1pt}
\def\ev{\kern-1pt\lowbkwd e{0.5}2\kern-1pt}
\def\vf{\kern-2pt\lowfwd f{2.5}2\kern-1pt}

\newtheorem{theorem}{Theorem}[section] 
\newtheorem{proposition}[theorem]{Proposition}

\newtheorem{lemma}[theorem]{Lemma}

\newtheorem{innercustomthm}{Theorem}

\newenvironment{customthm}[1]
  {\renewcommand\theinnercustomthm{#1}\innercustomthm}
  {\endinnercustomthm}

\theoremstyle{definition}
\newtheorem{example}[theorem]{Example}

\newtheorem{definition}[theorem]{Definition}

\theoremstyle{remark}

\newtheorem*{ack}{Acknowledgement}
\renewcommand{\cP}{\mathcal{P}}
\newcommand{\iecon}{infi\-nite\-ly edge-con\-nec\-ted}

\newcommand{\rminor}{\succcurlyeq}
\newcommand{\FG}{F}
\newcommand{\hFG}{\breve{F}}

\def\Pigraph{$\Pi$-graph}

\newcommand{\lin}{\normalfont\text{lin}}

\begin{document}
\title[The Farey graph is uniquely determined by its connectivity]{The Farey graph is uniquely determined\\by its connectivity}
\author{Jan Kurkofka}
\address{Universität Hamburg, Department of Mathematics, Bundesstraße 55 (Geomatikum), 20146 Hamburg, Germany}
\email{jan.kurkofka@uni-hamburg.de}
\begin{abstract}
We show that, up to minor-equivalence, the Farey graph is the unique minor-minimal graph that is infinitely edge-connected but such that every two vertices can be finitely separated.
\end{abstract}

\@namedef{subjclassname@2020}{\textup{2020} Mathematics Subject Classification}

\vspace*{-3cm}
\maketitle

\vspace*{-.7cm}
\begin{figure}[h]
    \centering
    \includegraphics[width=.4\textwidth]{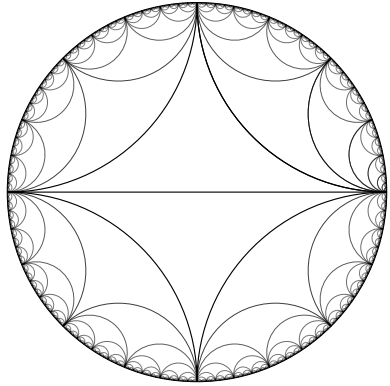}
    \caption{The Farey graph}
    \label{fig:FareyGraph}
\end{figure}

\vspace*{-.5cm}
\section{Introduction}

\noindent The Farey graph, shown in Figure~\ref{fig:FareyGraph} and surveyed in \cites{OfficeHoursGroupTheory,hatcher2017topology}, plays a role in a number of mathematical fields ranging from group theory and number theory to geometry and dynamics~\cite{OfficeHoursGroupTheory}.
Curiously, graph theory has not been among these until very recently, when it was shown that the Farey graph plays a central role in graph theory too: it is one of two infinitely edge-connected graphs that must occur as a minor in every infinitely edge-connected graph~\cite{TypicalInfinitelyEdgeconnectedGraphs}.
Infinite edge-connectivity, however, is only one aspect of the connectivity of the Farey graph, and it contrasts with a second aspect: the Farey graph does not contain infinitely many independent paths between any two of its vertices.
In this paper we show that the Farey graph
is uniquely determined by these two contrasting aspects of its connectivity: up to minor-equivalence, the Farey graph is the unique minor-minimal graph that is \iecon\ but such that every two vertices can be finitely separated.

A \emph{\Pigraph } is an \iecon\ graph that does not contain infinitely many independent paths between any two of its vertices.
A \Pigraph\ is \emph{typical} if it occurs as a minor in every \Pigraph .
Note that any two typical \Pigraph s are minors of each other; we call such graphs \emph{minor-equivalent}.
Our main result reads as follows:

\begin{customthm}{1}\label{Mainresult}
\hspace*{-1.9mm} Up to minor-equivalence, the Farey graph is the unique typical \mbox{\Pigraph }.
\end{customthm}

We shall see that there exist \Pigraph s that contain the Farey graph as a minor but are not minors of the Farey graph (Proposition~\ref{AtypicalPigraphs}).

Theorem~\ref{Mainresult} continues to hold if we require all minors to have finite branch sets; see Section~\ref{subsec:variations} and Theorem~\ref{MainresultFiniteBranch}.
This is best possible in the sense that one cannot replace `minors with finite branch sets' with `topological minors' (Proposition~\ref{BestPossible}).

This paper is organised as follows.
Section~\ref{Sec:Preliminaries} formally introduces the Farey graph.
In Section~\ref{sec:BestPossible} we prove results~\ref{AtypicalPigraphs}--\ref{BestPossible}.
We outline the overall strategy of the proof of Theorem~\ref{Mainresult} in Section~\ref{sec:strategy}.
We prepare the proof of Theorem~\ref{Mainresult} in Section~\ref{sec:GrainLines} and we prove Theorem~\ref{Mainresult} in Section~\ref{sec:FinalProof}.

\begin{ack}
I would like to thank my two referees for helpful comments.
\end{ack}

\section{Preliminaries}\label{Sec:Preliminaries}

\noindent We use the notation of Diestel's book~\cite{DiestelBook5}.
A~non-trivial path $P$ is an $A$-\emph{path} for a set $A$ of vertices if $P$ has its endvertices but no inner vertex in $A$.
Two $u$--$v$ paths are \emph{order-compatible} if they traverse their common vertices in the same order.
We write $G[X]$ for the subgraph of $G$ induced by the vertex set~$X$.
Given a path $P$ we write $\mathring{P}$ for the subpath obtained by deleting the endvertices of~$P$.
If $P$ starts in $u$ and ends in $v$, then we write $\mathring{u}P$ and $P\mathring{v}$ for $P-u$ and $P-v$ respectively.

The \emph{Farey graph} $F$ is the graph on $\Q\cup\{\infty\}$ in which two rational numbers $a/b$ and $c/d$ in lowest terms (allowing also $\infty=(\pm 1)/0$) form an edge if and only if $\det\bigl( \begin{smallmatrix}a & c\\ b & d\end{smallmatrix}\bigr)=\pm 1$, cf.~\cite{OfficeHoursGroupTheory}.
In this paper we do not distinguish between the Farey graph and the graphs that are isomorphic to it.
For our graph-theoretic proofs it will be more convenient to work with the following purely combinatorial definition of the Farey graph that is indicated in~\cite{OfficeHoursGroupTheory} and~\cite{hatcher2017topology}.

The \emph{halved Farey graph} $\hFG_0$ of order $0$ is a $K^2$ with its sole edge coloured blue.
Inductively, the \emph{halved Farey graph} $\hFG_{n+1}$ of order $n+1$ is the edge-coloured graph that is obtained from $\hFG_n$ by adding a new vertex $v_e$ for every blue edge $e$ of $\hFG_n$, joining each $v_e$ precisely to the endvertices of $e$ by two blue edges, and colouring all the edges of $\hFG_n\subset\hFG_{n+1}$ black.
The \emph{halved Farey graph} $\hFG:=\bigcup_{n\in\N}\hFG_n$ is the union of all $\hFG_n$ without their edge-colourings, and the \emph{Farey graph} is the union $F=G_1\cup G_2$ of two copies $G_1,G_2$ of the halved Farey graph such that $G_1\cap G_2=\hFG_0$.

\begin{lemma}\label{FareyGisMinorOfHalvedFareyG}
The halved Farey graph contains the Farey graph as a minor with finite branch sets.
\end{lemma}
\begin{proof}
If $e$ is a blue edge of $\hFG_1$, then the Farey graph is the contraction minor of $\hFG-e$ whose sole non-trivial branch set is~$V(\hFG_0)$, i.e., $(\hFG-e)/V(\hFG_0)\cong\FG$.
\end{proof}

\section{Atypical \texorpdfstring{$\Pi$}{Pi}-graphs and variations of the main result}\label{sec:BestPossible}

\noindent In this section we provide details on and prove the three results~\ref{AtypicalPigraphs}--\ref{BestPossible} that we briefly mentioned in the introduction.

\subsection{Atypical \texorpdfstring{$\Pi$}{Pi}-graphs}

Even though every \Pigraph\ contains the Farey graph as a minor by Theorem~\ref{Mainresult}, the converse is generally false:

\begin{proposition}\label{AtypicalPigraphs}
There is a countable planar \Pigraph\ that contains the Farey graph as an induced subgraph but is not a minor of the Farey graph.
\end{proposition}
\begin{proof}
The Farey graph is \emph{outerplanar} in that it has a drawing in which every vertex lies on the unit circle
and every edge is contained in the unit disc.
Thus, the graph obtained from the Farey graph by joining an additional vertex to all its vertices is still planar, and hence contains no $K_{3,3}$ minor.
As a consequence, the Farey graph does not contain a $K_{2,3}$ minor.

Let the graph $G$ be obtained from $K_{2,3}$ by adding for every edge a copy of the Farey graph such that the added copies are disjoint outside of $K_{2,3}$ and each copy intersects $K_{2,3}$ precisely in the edge for which it was added.
Then $G$ is a countable planar \Pigraph\ which contains the Farey graph as an induced subgraph, but its $K_{2,3}$ subgraph witnesses that $G$ cannot be a minor of the Farey graph.
\end{proof}

\subsection{Variations of the main result}\label{subsec:variations}

\noindent To prove Theorem~\ref{Mainresult} it suffices to show the following theorem.
A \emph{tight} minor is a minor with finite branch sets.

\begin{customthm}{2}\label{FGfinBranchSets}
Every \Pigraph\ contains the Farey graph as a tight minor.
\end{customthm}

\noindent Theorem~\ref{FGfinBranchSets} also implies the following variation of Theorem~\ref{Mainresult} where all minors are required to have finite branch sets.
Two graphs are \emph{tightly} minor-equivalent if they are tight minors of each other.
A~\Pigraph\ is \emph{tightly} typical if it occurs as a tight minor in every \Pigraph .

\begin{theorem}\label{MainresultFiniteBranch}
Up to tight minor-equivalence, the Farey graph is the unique tightly typical \Pigraph .\qed
\end{theorem}

\noindent This raises the question whether Theorem~\ref{Mainresult} continues to hold if we require all minors to be topological minors.
We answer this question in the negative:

\begin{proposition}\label{BestPossible}
There is a \Pigraph\ that contains the Farey graph as a tight minor but not as a topological minor.
\end{proposition}
\begin{proof}
By a recent result~\cite{OrderCompatiblePaths} there exists an \iecon\ graph $G$ that does not contain infinitely many edge-disjoint pairwise order-compatible paths between any two of its vertices; in particular, $G$ is a \Pigraph .
By Theorem~\ref{FGfinBranchSets}, the graph $G$ contains the Farey graph as a tight minor.
However, $G$ does not contain a subdivision of the Farey graph because the Farey graph contains infinitely many edge-disjoint pairwise order-compatible paths between any two of its vertices.
\end{proof}

\section{Overall proof strategy}\label{sec:strategy}

\noindent Our aim for the remainder of this paper is to prove Theorem~\ref{Mainresult}.
As we discussed in the previous section, to prove Theorem~\ref{Mainresult} it suffices to show that every \Pigraph\ contains the Farey graph as a minor with finite branch sets (Theorem~\ref{FGfinBranchSets}).
And by Lemma~\ref{FareyGisMinorOfHalvedFareyG} in turn it suffices to find a halved Farey graph minor with finite branch sets in any given \Pigraph .
The key idea of the proof is summarised in Theorem~\ref{GrainLineSplit} which states:

\emph{Suppose that $G$ is any subdivided \Pigraph\ and that $u,v$ are two distinct branch vertices of $G$.
Then there exist subgraphs $H_u,H_v\subset G$ whose vertex sets have non-empty intersection~$X$, such that the following conditions are satisfied:
\begin{enumerate}
\item $H_u[X]=H_v[X]$ and both induced subgraphs are finite and connected;
\item $X$ avoids $u$ and $v$;
\item both $H_u/X$ and $H_v/X$ are subdivided \Pigraph s in which $u,X$ and $v,X$ are branch vertices, respectively;
\item $uX$ is an edge of $H_u/X$ and $vX$ is an edge of $H_v/X$.
\end{enumerate}}

With this theorem at hand, it is straightforward to construct a halved Farey graph minor with finite branch sets in any given \Pigraph\ $G$:
Consider any edge $uv$ of $G$ and apply the theorem in $G$ to $u$ and $v$ to obtain subgraphs $H_u,H_v$ and a non-empty finite connected vertex set $X\subset V(G)$.
Then the three vertices $u,v$ and $X$ span a triangle $\hFG_1$ in $(H_u\cup H_v)/X$.
And since both $H_u/X$ and $H_v/X$ are subdivided \Pigraph s, we can reapply the theorem in $H_u/X$ to $u$ and $X$, and in $H_v/X$ to $v$ and $X$.
By iterating this process, we obtain a halved Farey graph minor with finite branch sets in the original graph $G$ at the limit, and this will complete the proof.
Therefore, it remains to prove Theorem~\ref{GrainLineSplit} on the one hand, and to use it to formally construct a halved Farey graph minor on the other hand.
In the next section, we prepare the proof of Theorem~\ref{GrainLineSplit}, and in the section after next we prove Theorem~\ref{GrainLineSplit} which we then use to prove Theorems~\ref{Mainresult} and~\ref{FGfinBranchSets}.

\section{Grain lines}\label{sec:GrainLines}

\noindent It is possible to prove Theorem~\ref{GrainLineSplit} from first principles.
In this paper, however, I~favour a more methodic proof.
The advantage of this proof is that it introduces a new tool, an $x$--$y$ grain line, that allows one to control infinite systems of edge-disjoint $x$--$y$ paths even when no two paths in the system are pairwise order-compatible.
In this section we introduce the concept of an $x$--$y$ grain line, we show that these exist whenever it matters (Theorem~\ref{GrainLineExistence}) and we show two lemmas that will help us prove Theorem~\ref{GrainLineSplit} using grain lines at the beginning of the next section.

Informally, we may think of an $x$--$y$ grain line as a pair $(L,\cP)$ where $\cP$ is a sequence of pairwise edge-disjoint $x$--$y$ paths $P_0,P_1,\ldots$ that need not be pairwise order-compatible but solve all incompatibilities at their linearly ordered `limit'~$L$. 
The limit $L$ will not be a graph-theoretic path but will be a linearly ordered set of vertices.
We remark, however, that it is possible to use the limit $L$ to define a topological $x$--$y$ path in a topological extension of any graph containing the grain line, see~\cite{KurkofkaMSc}*{§6.3}.

Here is the formal definition of an $x$--$y$ grain line:

\begin{definition}
An $x$--$y$ \emph{grain line} between two distinct vertices $x$ and $y$ is an ordered pair $(L,\cP)$ where $L=(L,{\le}_L)$ is a linearly ordered countable set of vertices with least element $x$ and greatest element~$y$, and $\cP=(P_n)_{n\in\N}$ is a sequence of pairwise edge-disjoint $x$--$y$ paths $P_n$, such that the following three conditions are satisfied:
\begin{enumerate}[label=(GL\arabic*)]
\item\label{GLinL} $L=\Big\{\,v\;\Big\vert\; \{\,n\in\N: v\in V(P_n)\,\}\text{ is a final segment of }\N\,\Big\}$;
\item\label{GLninL} if a vertex of a path $P_n$ is not contained in $L$, then it is not a vertex of any other path $P_m$ ($m\neq n$);
\item\label{GLgrain} for all $n\in\N$, the $x$--$y$ path $P_n$ and the linearly ordered vertex set $L$ induce the same linear ordering on the vertex set $L_{<n}:=L\cap\bigcup_{k<n}V(P_k)$. 
\end{enumerate}
\end{definition}

Note that we allow $L$ to be finite, for example we may have $L=\{x,y\}$ if the paths in~$\cP$ are independent.

We remark that~\ref{GLgrain} allows $P_n$ and $L$ to induce distinct linear orderings on the vertex set $V(P_n)\cap L$ if the inclusion $L_{<n}\subset V(P_n)\cap L$ is proper; in particular, $P_n$ and $P_{n+1}$ need not be order-compatible.
Allowing this becomes necessary, for example, if an \iecon\ graph does not contain infinitely many edge-disjoint pairwise order compatible paths between $x$ and $y$, see Example~\ref{WhirlExample}.

Clearly, $L=\bigcup_n L_{<n}$.
Note that if $(L,(P_n)_{n\in\N})$ is a grain line, then a vertex $v$ lies in $L$ if and only if it lies on all paths $P_n$ with $n\ge N$ for $N$ the first number with $v\in P_N$ if and only if it lies on at least two paths $P_n,P_m$ ($n\neq m$).
In particular,
\begin{align*}
V(P_n)\cap \bigcup_{k<n}V(P_k)=L_{<n}\text{ for all }n\in\N.
\end{align*}

We speak of an $x$--$y$ grain line $(L,(P_n)_{n\in\N})$ \emph{in} a graph $G$ if $\bigcup_{n\in\N} P_n\subset G$ (and hence $L\subset V(G)$).
Whenever a grain line is introduced as $(L,\cP)$, we tacitly assume $\cP=(P_n)_{n\in\N}$.
In general, however, we also allow sequences $\cP=(P_n)_{n\ge N}$ whose indexing starts at an arbitrary number $N>0$ in which case the definition of a grain line adapts in the obvious way.
We use the interval notation for $L$ as usual, i.e., we write $[\ell_1,\ell_2]_L=\{\,\ell\in L\mid \ell_1\le_L\ell\le_L\ell_2\,\}$ and so~on.

\begin{example}\label{GrainLinesFareyExample}
The blue Hamilton paths $P_n\subset\hFG_n$ are pairwise edge-disjoint and order-compatible, and hence give rise to an $x$--$y$ grain line in $\hFG$ for $x$ and $y$ the two vertices of $\hFG_0$.
In this case, $L=V(\hFG)$ is order-isomorphic to $\Q\cap [0,1]$.
\end{example}

\begin{example}\label{WhirlExample}
There exists an \iecon\ graph $G$ that does not contain infinitely many edge-disjoint pairwise order-compatible paths between any two of its vertices~\cite{OrderCompatiblePaths}; in particular, $G$ is a \Pigraph .
We shall see that the graph $G$ contains a grain line between any two of its vertices because it is \iecon ; see Theorem~\ref{GrainLineExistence} below.
However, since $G$ does not contain infinitely many edge-disjoint pairwise order-compatible paths between any two of its vertices, every grain line $(L,\cP)$ in $G$ has two paths $P_n$ and $P_{n+1}$ that are not order-compatible; in particular, $P_n$ induces the same linear ordering on $L_{<n}\subsetneq V(P_n)\cap L$ as $L$ does, but disagrees with $L$ on $V(P_n)\cap L=V(P_n)\cap V(P_{n+1})$ because $L$ induces the ordering of $P_{n+1}$ on $V(P_n)\cap V(P_{n+1})=L_{<n+1}$.
%
This is why we do not strengthen \ref{GLgrain} to require that $P_n$ and $L$ induce the same linear ordering on $V(P_n)\cap L\supset L_{<n}$.
\end{example}



Our first result on grain lines shows that they exist whenever it matters:

\begin{theorem}\label{GrainLineExistence}
Let $x$ and $y$ be any two distinct vertices of a graph $G$.
Then there exists an $x$--$y$ grain line in $G$ if and only if $G$ contains infinitely many edge-disjoint $x$--$y$ paths.
\end{theorem}

In the proof we employ inverse systems, and for the sake of convenience we dedicate a paragraph to their definition.

A partially ordered set $(I,\le)$ is said to be \emph{directed} if for every two $i,j\in I$ there is some $k\in I$ with $k\ge i,j$.
Let $(\,X_i\mid i\in I\,)$ be a family of finite sets indexed by some directed poset $(I,\le)$.
Furthermore, suppose that we are given a family $(\,\varphi_{ji}\colon X_j\to X_i\,)_{i\le j\in I}$ of mappings which are the identity on $X_i$ in case of $i=j$ and which are \emph{compatible} in that $\varphi_{ki}=\varphi_{ji}\circ\varphi_{kj}$ for all $i\le j\le k$. 
Then both families together are said to form an \emph{inverse system} (\emph{of finite sets}), and the maps $\varphi_{ji}$ are called its \emph{bonding maps}. 
We denote such an inverse system by $\{X_i,\varphi_{ji},I\}$ or $\{X_i,\varphi_{ji}\}$ for short if $I$ is clear from context.
Its \emph{inverse limit} $\invLim{} X_i=\invLim{}(\,X_i\mid i\in I\,)$ is the set
\begin{align*}
\invLim{} X_i=\{\,(x_i)_{i\in I}\mid \varphi_{ji}(x_j)=x_i\text{ for all }j\ge i\,\}\subseteq \prod_{i\in I}X_i.
\end{align*}
If every $X_i$ is non-empty, then the inverse limit $\invLim{}X_i$ is non-empty as well.
For more details on inverse systems and their more general definition for topological spaces, see \cite{EngelkingBook} or~\cite{ProfiniteGroups}.

\begin{proof}[{Proof of Theorem~\ref{GrainLineExistence}}]
Every $x$--$y$ grain line comes with a system of infinitely many edge-disjoint $x$--$y$ paths.
For the backward implication let $x$ and $y$ be given, and let $\cQ$ be any countably infinite collection of edge-disjoint $x$--$y$ paths in~$G$.
Moreover, we let $\cX$ be the collection of all finite subsets of the vertex set of the subgraph $\bigcup\cQ\subset G$, directed by inclusion.

Given $X\in\cX$ we write $\lin (X)$ for the finite collection of all linearly ordered subsets of $X$.
Letting, for all $X\subset X'\in\cX$, the map $\varphi_{X',X}\colon\lin(X')\to\lin(X)$ take every linearly ordered subset of $X'$ to its restriction with respect to $X$ turns the finite sets $\lin(X)$ into an inverse system $\{\,\lin(X),\,\varphi_{X',X},\,\cX\,\}$.

Every $x$--$y$ path $P\in\cQ$ naturally induces a linear ordering ${\le}_P$ on its vertex set with $x<_P y$, and for every $X\in\cX$ we denote by ${\le}_P^X$ the linear ordering on $V(P)\cap X$ induced by ${\le}_P$.
Then for every $X\in\cX$ we define a map $\psi_X\colon\cQ\to\lin(X)$ by letting
\begin{align*}
\psi_X(P):=(V(P)\cap X,{\le}_P^X)
\end{align*}
for all $P\in\cQ$, and we put
\begin{align*}
\cL_X:=\{\,\xi\in\lin(X)\mid\psi_X^{-1}(\xi)\subset\cQ\text{ is infinite}\,\}
\end{align*}
noting that $\cL_X\subset\lin(X)$ is non-empty by the pigeonhole principle.
Since the maps $\psi_X$ commute with the bonding maps $\varphi_{X',X}$ as pictured in the diagram below,
\begin{equation*}
\begin{tikzcd}
  & \cQ\arrow[dl, "\psi_X"']\arrow[dr, "\psi_{X'}"]\\\
\lin (X) & & \lin (X')\arrow[ll, "\varphi_{X',X}"']
\end{tikzcd}
\end{equation*}
the restrictions of these bonding maps to the sets $\cL_X$ yield another inverse system, namely $\{\,\cL_X,\,\varphi_{X',X}\rest\cL_{X'},\,\cX\,\}$.
And as the finite sets $\cL_X$ are all non-empty, this inverse system has an element $(\,(L_X,{\le_X})\mid X\in\cX\,)$ in its limit.

Finally, we define an $x$--$y$ grain line $(L,\cP)$, as follows.
We let $L:=\bigcup_{X\in\cX}L_X$ and ${\le}_L:=\bigcup_{X\in\cX}{\le}_X$.
To obtain $\cP=(P_n)_{n\in\N}$ we choose pairwise edge-disjoint $x$--$y$ paths $P_0,P_1,\ldots$ from $\cQ$ inductively, as follows.
At step 0, we let $X_0:=\emptyset$ and choose $P_0\in\psi_{X_0}^{-1}(L_{X_0})$ arbitrarily (we abbreviate $L_X=(L_X,{\le}_X)$).
At step $n+1$, we let $X_{n+1}:=X_n\cup V(P_n)$ and we pick from the infinite preimage $\psi_{X_{n+1}}^{-1}(L_{X_{n+1}})$ a path $P_{n+1}$ other than the previously chosen paths $P_0,\ldots,P_n$.
Then the paths $P_0,P_1,\ldots$ are pairwise edge-disjoint. To verify that $(L,\cP)$ is an $x$--$y$ grain line in~$G$, it remains to check \ref{GLinL}--\ref{GLgrain}.

First, we present an argument which implies \ref{GLinL}~and~\ref{GLninL}.
Consider any vertex $v$ of $\bigcup\cP$ and let $N\in\N$ be minimal with $v\in V(P_N)$.
Then $v\in X_{N+1}$.
If~$v$ is contained in the set~$L_{X_{N+1}}$, then $v$ is also contained in~$L$ and in all sets $L_{X_n}$ with $n\ge N+1$, and hence $v$ is also contained in all paths $P_n$ with $n\ge N+1$.
Otherwise, if $v$ is not contained in the set $L_{X_{N+1}}$, then $v$ is contained in no set $L_X$ with $X\in\cX$ because $L_X\cap X_{N+1}$ is included in $L_{X\cup X_{N+1}}\cap X_{N+1}=L_{X_{N+1}}$.
So in particular $v$ is not contained in~$L$.
And $v$ is not contained in any path $P_n$ with $n\ge N+1$, because all these $P_n$ intersect $X_{N+1}\subset X_n$ precisely in $L_{X_{N+1}}$.

Finally, \ref{GLgrain} follows immediately from the two facts that, for all $n\in\N$,
\[L_{<n}=L\cap\bigcup_{k< n}V(P_k)\subset X_{n}\]
and $P_{n}\in\psi^{-1}_{X_{n}}(L_{X_{n}})$.
\end{proof}

A grain line $(L,\cP)$ is \emph{wild} if $L$ is order-isomorphic to $\Q\cap [0,1]$.
We call a grain line $(L,\cP)$ \emph{wildly presented} if, for every $n\in\N$, whenever $\ell_1<_L\ell_2$ are elements of $L_{<n}\subset L$ then $\mathring{\ell}_1 P_n\mathring{\ell}_2$ has a vertex in $(\ell_1,\ell_2)_L$.
The grain line in Example~\ref{GrainLinesFareyExample} is both wild and wildly presented.
Wildly presented grain lines are wild.
Conversely, if a grain line $(L,\cP)$ is wild, then $\cP=(P_n)_{n\in\N}$ has a subsequence $(P_{n_k})_{k\in\N}$ such that $(L,(P_{n_k})_{k\in\N})$ is wildly presented.

\begin{lemma}\label{wildGLinPigraph}
Every grain line in a subdivided \Pigraph\ is wild; in particular, in a subdivided \Pigraph\ every grain line can be chosen to be wildly presented.
\end{lemma}

In the proof we use the following properties of grain lines.
Given a grain line $(L,\cP)$ we say that a path $P_n$ does $(L,\cP)$-\emph{grain} 
a set $U$ of vertices if, for all $m\ge n$, we have $V(P_m)\cap U=L\cap U$ and the path $P_m$ induces the same linear ordering on this intersection as $L$ does.
If $(L,\cP)$ is clear from context, we also say that $P_n$ \emph{grains} $U$.
Every path $P_n$ grains the union $\bigcup_{k<n}V(P_k)$ by~\ref{GLgrain}.
And for every finite vertex set $X$ there is a number $n\in\N$ such that $P_n$ grains $X$.
We will use this latter property frequently in the proofs to come.

\begin{proof}[{Proof of Lemma~\ref{wildGLinPigraph}}]
Suppose that $(L,\cP)$ is any grain line in some given subdivided \Pigraph ~$G$.
It suffices to show that $(L,\cP)$ is wild.
For this, consider any two elements $\ell_1,\ell_2\in L$ with $\ell_1<_L\ell_2$.
Then $\ell_1$ and $\ell_2$ must have infinite degree in~$G$; in particular, $\ell_1$ and $\ell_2$ must be branch vertices of $G$.
Since $G$ is a subdivided \Pigraph , we find a finite vertex set $S\subset V(G)\setminus\{\ell_1,\ell_2\}$ that separates $\ell_1$ and $\ell_2$ in $G-\ell_1\ell_2$ (the edge $\ell_1\ell_2$ need not exist).
Then we pick $N\in\N$ such that $P_N$ avoids the edge $\ell_1\ell_2$ and grains the finite vertex set $S\cup\{\ell_1,\ell_2\}$.
Now $\ell_1 P_N\ell_2$ must meet $S$ in a vertex $s$, and then $P_N$ graining $S\cup\{\ell_1,\ell_2\}$ implies $s\in L$ with $\ell_1<_L s<_L\ell_2$ as desired.
\end{proof}

Grain lines can be restricted such that the restriction is again a grain line, and restricting a grain line preserves wild presentations:

\begin{lemma}\label{grainLineRestriction}
If $(L,\cP)$ is a grain line with $\ell_1<_L\ell_2$ and $N\in\N$ is such that $P_N$ grains $\{\ell_1,\ell_2\}$, then $([\ell_1,\ell_2]_L,(\ell_1 P_n\ell_2)_{n\ge N})$ is an $\ell_1$--$\ell_2$ grain line that is wildly presented if $(L,\cP)$ is.
\end{lemma}
\begin{proof}
First, we show that $([\ell_1,\ell_2]_L,(\ell_1 P_n\ell_2)_{n\ge N})$ is an $\ell_1$--$\ell_2$ grain line.

\ref{GLinL} We have to show the equality
\begin{align*}
[\ell_1,\ell_2]_L=\Big\{\,v\;\Big\vert\;\{\,n\in\N_{\ge N}:v\in V(\ell_1 P_n\ell_2)\,\}\text{ is a final segment of }\N_{\ge N}\,\Big\}.
\end{align*}
We start with the backward inclusion.
If a vertex $v$ lies on $\ell_1 P_n\ell_2$ for all $n$ in some final segment of $\N_{\ge N}$ then it lies in $L$ by \ref{GLninL} for $(L,\cP)$, and in particular it also lies on $\ell_1 P_n\ell_2$ when $P_n$ does $(L,\cP)$-grain $\{\ell_1,v,\ell_2\}$ so $v\in [\ell_1,\ell_2]_L$ follows.
Conversely, if $v$ is a vertex in $[\ell_1,\ell_2]_L$ and $k\ge N$ is minimal with $v\in\ell_1 P_k\ell_2$, then $P_{k+1}$ does $(L,\cP)$-grain $\{\ell_1,v,\ell_2\}$.
Therefore, $v$ is contained in $\ell_1 P_n\ell_2$ for all $n\ge k$, and hence $\N_{\ge k}$ witnesses that $v$ is contained in the right hand side of the equation.

\ref{GLninL} Consider any vertex $v\in (\bigcup_{n\ge N}\ell_1 P_n\ell_2)-[\ell_1,\ell_2]_L$ and let $k\ge N$ be minimal such that $\ell_1 P_k\ell_2$ contains $v$.
If $v$ is not contained in $L$, then $P_k$ is the only path from $\cP$ containing $v$, and hence $\ell_1 P_k\ell_2$ is the only path from $(\ell_1 P_n\ell_2)_{n\ge N}$ containing $v$.
Otherwise $v$ is contained in $L\setminus [\ell_1,\ell_2]_L$ so, say, $\ell_2<_L v$.
Then, as $P_n$ with $n>k$ does $(L,\cP)$-grain $V(P_k)$, the vertex $\ell_2$ precedes $v$ on $P_n$, giving $v\notin\ell_1 P_n\ell_2$ as desired.

\ref{GLgrain} 
Consider any $n\ge N$ and write $L'_{<n}:=[\ell_1,\ell_2]_L\cap\bigcup_{k=N}^{n-1}V(\ell_1 P_k\ell_2)$.
By the already shown \ref{GLinL} we have $L'_{<n}\subset V(\ell_1 P_n\ell_2)$, so $\ell_1 P_n\ell_2$ does induce a linear ordering on $L'_{<n}$, and it coincides with the linear ordering induced by $[\ell_1,\ell_2]_L$ by \ref{GLgrain} for $(L,\cP)$.

Therefore, $([\ell_1,\ell_2]_L,(\ell_1 P_n\ell_2)_{n\ge N})$ is an $\ell_1$--$\ell_2$ grain line; now we show that it is wildly presented if $(L,\cP)$ is.
For this consider any $n\ge N$ with some two elements $\ell<_L\ell'$ of $L'_{<n}$.
Then, as $(L,\cP)$ is wildly presented and $L'_{<n}\subset L_{<n}$, the subpath $\mathring{\ell}P_n\mathring{\ell}'$ of $\ell_1 P_n\ell_2$ has a vertex in $(\ell,\ell')_L$.
\end{proof}

\section{Proof of the main result}\label{sec:FinalProof}

\noindent In this section, we employ our results on grain lines to prove Theorem~\ref{GrainLineSplit}, which we then use to prove Theorems~\ref{Mainresult} and~\ref{FGfinBranchSets}.


\begin{theorem}\label{GrainLineSplit}
Suppose that $G$ is any subdivided \Pigraph\ and that $u,v$ are two distinct branch vertices of $G$.
Then there exist subgraphs $H_u,H_v\subset G$ whose vertex sets have non-empty intersection~$X$, such that the following conditions are satisfied:
\begin{enumerate}
\item $H_u[X]=H_v[X]$ and both induced subgraphs are finite and connected;
\item $X$ avoids $u$ and $v$;
\item both $H_u/X$ and $H_v/X$ are subdivided \Pigraph s in which $u,X$ and $v,X$ are branch vertices, respectively;
\item $uX$ is an edge of $H_u/X$ and $vX$ is an edge of $H_v/X$.
\end{enumerate}
\end{theorem}

\begin{proof}
Without loss of generality we may assume that $uv$ is not an edge of $G$.
Using that $G$ is a subdivided \Pigraph\ we find a finite vertex set $S\subset V(G)\setminus\{u,v\}$ that separates $u$ and $v$ in $G$.
We write $C_u$ and $C_v$ for the two distinct components of $G-S$ that contain $u$ and $v$ respectively.
Next, we use Theorem~\ref{GrainLineExistence} and Lemma~\ref{wildGLinPigraph} to find a wildly presented $u$--$v$ grain line $(L,\cP)$ in $G$.
Without loss of generality we may assume that $P_0$ grains the finite vertex set $S$.
We let $s_u$ be the first vertex of the $u$--$v$ path $P_0$ in $S$, and we let $s_v$ be the last vertex of $P_0$ in $S$.
That is to say that $s_u$ and $s_v$ are the least and greatest vertex of $L$ in~$S$.
Then, for all $n\in\N$, the paths $uP_n s_u$ and $s_vP_n v$ are contained in $G[C_u+s_u]$ and $G[s_v+C_v]$ respectively.

Next, we let $x_u$ and $x_v$ be the least and greatest vertex of $L$ in $V(\mathring{P_0})$. 
Moreover, we let $L_u:=[u,x_u]_L$ and $\cP_u:=(u P_n x_u)_{n\ge 1}$, and we let $L_v:=[x_v,v]_L$ and $\cP_v:=(x_v P_n v)_{n\ge 1}$.
Then $(L_u,\cP_u)$ and $(L_v,\cP_v)$ are wildly presented \mbox{$u$--$x_u$} and \mbox{$x_v$--$v$} grain lines in $G$ by Lemma~\ref{grainLineRestriction}.
We claim that $H_u:=P_0\mathring{v}\cup\bigcup\cP_u$ and $H_v:=\mathring{u}P_0\cup\bigcup\cP_v$ are the desired subgraphs.

First, we show that $X=V(\mathring{P_0})$ and that $X$ satisfies~(i), (ii) and~(iv).
For this, it suffices to show that for every $n\ge 1$ the paths $u P_n x_u$ and $x_v P_n v$ are $u$--$\mathring{P_0}$ and $\mathring{P_0}$--$v$ paths in $G[C_u+s_u]$ and $G[s_v+C_v]$, respectively.
The vertex $s_u\in L\cap S\subset L\cap V(\mathring{P_0})$ was a candidate for~$x_u$, implying $x_u\le_L s_u$, and then for all $n\ge 1$ the path $P_n$ graining $V(P_0)$ gives $u P_n x_u\subset u P_n s_u\subset G[C_u+s_u]$ on the one hand and that $x_u$ is the first vertex of $P_n$ in $\mathring{P_0}$ on the other hand; for the paths $x_vP_n v$ we employ symmetry.

(iii) follows from the facts that $(L_u,\cP_u)$ and $(L_v,\cP_v)$ are wildly presented and that all paths $u P_n x_u$ and $x_v P_n v$ ($n\ge 1$) are $u$--$\mathring{P_0}$ and $\mathring{P_0}$--$v$ paths respectively.
\end{proof}

Now we have almost all we need to prove Theorems~\ref{Mainresult} and~\ref{FGfinBranchSets}.
In the proof of Theorem~\ref{FGfinBranchSets}, we will face the construction of a minor with finite branch sets in countably many steps.
The following notation and lemma will help us to keep the technical side of this construction to the minimum.

Suppose that $G$ and $H$ are two graphs with $H$ a minor of $G$.
Then there are a vertex set $U\subset V(G)$ and a surjection $f\colon U\to V(H)$ such that the preimages $f^{-1}(x)\subset U$ form the branch sets of a model of $H$ in $G$.
A \emph{minor-map} $\varphi\colon G\rminor H$ formally is such a pair $(U,f)$.
Given $\varphi=(U,f)$ we address $U$ as $V(\varphi)$ and we write $\varphi=f$ by abuse of notation.
Usually, we will abbreviate `minor-map' as `map'.

\begin{lemma}\label{limitMinor}
Let $G_0,G_1,\ldots$ and $H_0\subset H_1\subset\cdots$ be two sequences of graphs \mbox{$H_n\subset G_n$} with maps $\varphi_n\colon G_n\rminor G_{n+1}$ such that for every vertex $x\in G_{n+1}$ the preimage $\varphi_n^{-1}(x)$ is finite if $x\notin H_{n}$ and equal to $\{x\}$ if $x\in H_{n}$.
Then $G_0$ contains $\bigcup_{n\in\N}H_n$ as a minor with finite branch sets.
\end{lemma}
\begin{proof}
The proof of~\cite{TypicalInfinitelyEdgeconnectedGraphs}*{Lemma~5.12} shows this.
\end{proof}


\begin{proof}[{Proof of Theorem~\ref{FGfinBranchSets}}]
Let $G$ be any \Pigraph . We have to find a Farey graph minor in $G$ with finite branch sets.
By Lemma~\ref{FareyGisMinorOfHalvedFareyG} it suffices to find a halved Farey graph minor with finite branch sets in $G$.

Call a graph a \emph{foresighted} halved Farey graph of order $n\in\N$ if it is the edge-disjoint union of $\hFG_n$ with subdivided \Pigraph s $A_{uv}$, one for every blue edge $uv\in \hFG_n$, such that:
\begin{itemize}[label={\textbf{--}}]
    \item each $A_{uv}$ meets $\hFG_n$ precisely in $u$ and $v$ but $uv\notin A_{uv}$;
    \item $u$ and $v$ are branch vertices of $A_{uv}$;
    \item every two distinct $A_e$ and $A_{e'}$ meet precisely in the intersection $e\cap e'$ of their corresponding edges (viewed as vertex sets).
\end{itemize}

\noindent To find a halved Farey graph minor with finite branch sets in $G$, it suffices by Lemma~\ref{limitMinor} to find a sequence $G=:G_0,G_1,\ldots$ of foresighted halved Farey graphs of orders $0,1,\ldots$ with maps $\varphi_n\colon G_n\rminor G_{n+1}$ such that $\varphi_n^{-1}(x)$ is finite for all $x\in G_{n+1}-\hFG_{n}$ and $\varphi_n^{-1}(x)=\{x\}$ for all $x\in\hFG_{n}$.

To get started, pick any edge $e$ of $G$, and note that $G=G_0$ is a foresighted halved Farey graph of order $0$ with $A_e=G-e$ when we rename $e$ to the edge of which $\hFG_0=K^2$ consists.

At step $n+1$ suppose that we have already constructed $G_n\supset\hFG_n$ and consider the subdivided \Pigraph s $A_e$ that were added to $\hFG_n$ to form~$G_n$.
Theorem~\ref{GrainLineSplit} yields in each $A_e$ two subgraphs $H_u^e,H_v^e$ for $e=uv$ whose vertex sets have non-empty intersection~$X^e$, such that the following conditions are satisfied:
\begin{enumerate}
    \item $H_u^e[X^e]=H_v^e[X^e]$ and both induced subgraphs are finite and connected;
    \item $X^e$ avoids $u$ and $v$;
    \item both $H_u^e/X^e$ and $H_v^e/X^e$ are subdivided \Pigraph s in which $u,X^e$ and $v,X^e$ are branch vertices, respectively;
    \item $uX^e$ is an edge of $H_u^e/X^e$ and $vX^e$ is an edge of $H_v^e/X^e$.
\end{enumerate}
We obtain $G_{n+1}$ from $G_n$ by replacing each $A_e$ with the minor $(H_u^e\cup H_v^e)/X^e$ and renaming the vertices $X^e$ to~$v_e$ (recall that $v_e$ is the vertex of~$\hFG_{n+1}-\hFG_n$ that arises from $e=uv\in\hFG_n$ in the recursive definition of~$\hFG_{n+1}$).
Then all triples $u,v,v_e$ span a triangle in $G_{n+1}$, giving $\hFG_{n+1}\subset G_{n+1}$, and the map $\varphi_n\colon G_n\rminor G_{n+1}\supset\hFG_{n+1}$ gets defined as desired (the only non-trivial fibres are of the form~$\varphi_n^{-1}(v_e)=X^e$, and all sets $X^e$ are finite).
To see that $G_{n+1}$ is a halved Farey graph of order $n+1$, note that $G_{n+1}$ arises from $\hFG_{n+1}$ by adding for all blue edges $uv_e$ and $v_e v$ of $\hFG_{n+1}$ the subdivided \Pigraph s $A_{uv_e}$ and $A_{v_e v}$ obtained from $H_u^e/X^e$ and $H_v^e/X^e$ by deleting the edges $uX^e$ and $vX^e$ respectively and renaming the vertex $X^e$ to $v_e$.
This completes the proof.
\end{proof}

\begin{proof}[{Proof of Theorem~\ref{Mainresult}}]
Theorem~\ref{FGfinBranchSets} implies Theorem~\ref{Mainresult}.
\end{proof}

\end{document}